\documentclass[10pt]{article}
\font\smallit=cmti10
\font\smalltt=cmtt10

\usepackage{amssymb,latexsym,amsmath,epsfig,amsthm}

\makeatletter

\renewcommand\section{\@startsection {section}{1}{\z@}
{-30pt \@plus -1ex \@minus -.2ex}
{2.3ex \@plus.2ex}
{\normalfont\normalsize\bfseries\boldmath}}

\renewcommand\subsection{\@startsection{subsection}{2}{\z@}
{-3.25ex\@plus -1ex \@minus -.2ex}
{1.5ex \@plus .2ex}
{\normalfont\normalsize\bfseries\boldmath}}

\renewcommand{\@seccntformat}[1]{\csname the#1\endcsname. }

\makeatother

\newtheorem{theorem}{Theorem}
\newtheorem{lem}{Lemma}

\theoremstyle{definition}
\newtheorem{definition}{Definition}

\usepackage{verbatim}
\usepackage{tikz}
\usepackage{xcolor}
\usepackage{pgfplots}
\usepackage{graphicx}
\usepackage{hyperref}
\usepackage{algpseudocode}
\usepackage{algorithm}
\usepackage{mathtools}
\usepackage{listings}
\usepackage{blindtext}
\usepackage{titlesec}
\usepackage{array}
\usepackage{makecell}
\usepackage{float}
\usepackage[utf8]{inputenc}

\DeclarePairedDelimiter{\ceil}{\lceil}{\rceil}

\graphicspath{ {./images/} }
\usetikzlibrary{shapes, arrows, calc, arrows.meta, fit, positioning}
\tikzset{  
    auto,node distance =1.5 cm and 1.3 cm, thick,% node distance is the distance between one node to other, where 1.5cm is the length of the edge between the nodes  
    state/.style ={circle, draw, minimum width = 0.9 cm, fill=blue!20}, % the minimum width is the width of the ellipse, which is the size of the shape of vertex in the node graph  
    point/.style = {circle, draw, inner sep=0.18cm, fill, node contents={}},  
    el/.style = {inner sep=2.5pt, align=right, sloped}  
}

\begin{document}

\begin{center}
\uppercase{\bf Ramsey Theory on the Integer Grid: The ``L" Problem}
\vskip 20pt
{\bf Isaac Mammel}\\
{\smallit Department of Mathematics, University of Maryland, College Park, Maryland, United States}\\
{\tt imammel@terpmail.umd.edu}\\
\vskip 10pt
{\bf William Smith}\\
{\smallit Department of Mathematics, University of South Carolina, Columbia, South Carolina, United States}\\
{\tt wjs11@email.sc.edu}\\
\vskip 10pt
{\bf Carl Yerger}\\
{\smallit Department of Mathematics and Computer Science, Davidson College, Davidson, North Carolina, United States}\\
{\tt cayerger@davidson.edu}\\
\end{center}
\vskip 20pt

\centerline{\smallit Received: , Revised: , Accepted: , Published: } % We will fill in the dates
\vskip 30pt

\centerline{\bf Abstract}
\noindent
In an $[n] \times [n]$ integer grid, a monochromatic $L$ is any set of points $\{(i, j), (i, j+t), (i+t, j+t)\}$ for some positive integer $t$, where $1 \leq i, j, i+t, j+t \leq n$. In this paper, we investigate the upper bound for the smallest integer $n$ such that a $3$-colored $n \times n$ grid is guaranteed to contain a monochromatic $L$. We use various methods, such as counting intervals on the main diagonal and using Golomb rulers, to improve the upper bound. This bound originally sat at 2593, and we improve it first to 1803, then to 1573, then to 772, and finally to 493. In the latter part of this paper, we discuss the lower bound and our attempts to improve it using SAT solvers.

\pagestyle{myheadings}
\markright{\smalltt INTEGERS: 24 (2024)\hfill}
\thispagestyle{empty}
\baselineskip=12.875pt
\vskip 30pt

\section{Introduction to the Problem}
\hspace{0.2in} The problem we deal with in this paper is a Ramsey-type problem on the integer grid. Namely, we deal with the following corollary of the Gallai-Witt theorem. \begin{theorem}
    [\cite{witt}] For all positive integers $k$, there exists a positive integer $n$ such that, for all $k$-colorings of $[n] \times [n]$, there is a monochromatic $L$. That is, there exist positive integers $x, y,$ and $d$ such that:
    \begin{enumerate}
        \item $(x, y), (x+d, y),$ and $(x+d, y + d)$ are all in $[n] \times [n]$, and
        \item $(x, y), (x+d, y),$ and $(x+d, y + d)$ are all the same color.
    \end{enumerate}
\end{theorem} To elaborate, an $[n] \times [n]$ \textit{integer grid} is an integer lattice with $n$ rows, $n$ columns, and a set of points such that, for each pair of integers $i, j$ with $1 \leq i, j \leq n$, there is exactly one point $p$ lying in row $i$ and column $j$. In such a grid, the rows are counted going downward and the columns are counted left-to-right. Such a point $p$ is said to be \textit{located at} $(i,j)$. 

An interesting problem is to find $n = R_k(L)$, the least positive integer such that the $n \times n$ integer grid with $k$ colors must contain a monochromatic $L$. From our definition, it is obvious that $R_1(L) = 2$, since a monochromatic $L$ can only appear in a grid of at least length $2$ and is guaranteed to appear in a monochromatic $2 \times 2$ grid at the points $(1,1), (2, 1),$ and $(2, 2)$. It is also known that $R_2(L)=5$ \cite{manske}. As a result, we focus mostly on $R_3(L)$, of which much less is known.

In 2015, Canacki et al. \cite{canackisat} found that $21 \leq R_3(L) \leq 2593$. This paper improves substantially on the upper bound, the methods of which are detailed in the following sections. In Section 2.1, we prove $R_3(L) \leq 1804$, in Section 2.2, we lower this upper bound to $1573$, and in Section $2.3$, we lower this bound to $772$ and finally $493$. As we do not improve upon the lower bound, discussion of it is omitted from this paper. However, information on the lower bound as well as speculation on how to improve it can be found in the arxiv version of this paper. 

\section{The Upper Bound}
To aid in our arguments, we will open with the proof in Canacki et al. \cite{canackisat} that $R_3(L) \leq 2593$. However, before continuing, we will introduce diagonals and subdiagonals, as these are crucial to all our proofs. The \textit{main diagonal} on an $n \times n$ grid is the series of points $(1, 1), (2, 2), \dots, (n, n)$. A \textit{subdiagonal} is a series of points in the $n \times n$ integer grid that follow the form of either $(1, k), (2, k+1), \dots, (1+n-k, n)$ or $(k, 1), (k+1, 2), \dots, (n, 1+n-k)$ for some integer $k$ where $1 \leq k \leq n$. We label $S_k$ as the subdiagonal containing the point $(1, k+1)$ (that is, the subdiagonal $k$ points below the main diagonal). The subdiagonals above the main diagonal are rarely discussed.

In the proofs that follow, we assume our $3$ colors in a $3$-colored grid to be red, green, and blue.

\begin{theorem}
    [\cite{canackisat}] The value of $R_3(L)$ is at most $2593$.
\label{2593}\end{theorem}

\begin{proof}
    Assume we have a $3$-coloring of the $n \times n$ grid. Consequentially, there are $n$ points on the main diagonal, and since the diagonal is $3$-colored, there exists $\frac{n}{3}$ points of a single color on the main diagonal. Without loss of generality, assume this color is red. For each pair of red points on the main diagonal, there is a unique point in the grid below this main diagonal such that if colored red, this point and the two selected red points form a monochromatic $L$. Thus, there are $\binom{n/3}{2}$ points in the grid that must be colored either blue or green.
    
    As there are $\binom{n/3}{2}$ blue or green points across $n-1$ subdiagonals, there lie $\frac{\binom{n/3}{2}}{n-1}$ blue or green points on some subdiagonal $S_{\mathcal{B}}$. Since these points are $2$-colored, there are either $\frac{\binom{n/3}{2}}{2(n-1)}$ blue or $\frac{\binom{n/3}{2}}{2(n-1)}$ green points on this diagonal. Without loss of generality, assume the majority are blue, and let us define $b = \frac{\binom{n/3}{2}}{2(n-1)}$. As with the red points, each pair of these blue points corresponds to a unique point in the subdiagonals below such that if colored blue, this point and the pair of blue points form a monochromatic $L$. Moreover, this point cannot be colored red, since it will form a monochromatic $L$ with the red points that force this pair of blue points to be blue. So there are $\binom{b}{2}$ points that must be colored green.
    
    Since there are at most $n-2$ subdiagonals below $S_{\mathcal{B}}$, one of these subdiagonals $S_{\mathcal{G}}$ must contain $\frac{\binom{b}{2}}{n-2}$ green points. For each of these green points on $S_{\mathcal{G}}$, there is a corresponding point that cannot be colored green, or else it will form a monochromatic $L$ with its pair on $S_{\mathcal{G}}$. Moreover, it cannot be colored red or blue, or it will form a monochromatic $L$ with the red and/or blue points that forced the pair on $S_{\mathcal{G}}$ to be green. So if there is more than one green point on $S_{\mathcal{G}}$ (that is, if $\frac{\binom{b}{2}}{n-2} > 1$), we reach a contradiction. The smallest $n$ such that $\frac{\binom{b}{2}}{n-2} > 1$ is $2593$ \cite{canackisat}.
\end{proof}

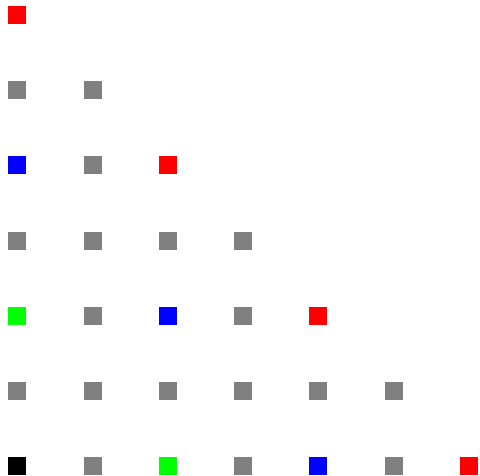
\begin{figure}[!h]
\centering
\begin{tikzpicture}
% a is the name of the node and A is the text inside the node/vertex  
    \node[fill=red] (1) at (0,6) {}; 
    \node[fill=gray] (2) at (0,5) {};
    \node[fill=blue] (3) at (0,4) {};  
    \node[fill=gray] (4) at (0, 3) {};  
    \node[fill=green] (5) at (0, 2) {}; 
    \node[fill=gray] (6) at (0, 1) {}; 
    \node[fill=black] (7) at (0, 0) {}; 
    \node[fill=gray] (8) at (1,5) {};
    \node[fill=gray] (10) at (1, 4) {}; 
    \node[fill=gray] (11) at (1, 3) {}; 
    \node[fill=gray] (12) at (1, 2) {}; 
    \node[fill=gray] (13) at (1, 1) {}; 
    \node[fill=gray] (14) at (1, 0) {}; 
    \node[fill=red] (16) at (2,4) {};
    \node[fill=gray] (17) at (2, 3) {}; 
    \node[fill=blue] (18) at (2, 2) {}; 
    \node[fill=gray] (19) at (2, 1) {}; 
    \node[fill=green] (20) at (2, 0) {};
    \node[fill=gray] (22) at (3,3) {};
    \node[fill=gray] (23) at (3, 2) {}; 
    \node[fill=gray] (24) at (3, 1) {}; 
    \node[fill=gray] (25) at (3, 0) {}; 
    \node[fill=red] (27) at (4,2) {};
    \node[fill=gray] (28) at (4, 1) {};
    \node[fill=blue] (29) at (4, 0) {};
    \node[fill=gray] (31) at (5,1) {};
    \node[fill=gray] (32) at (5, 0) {};
    \node[fill=red] (34) at (6,0) {};
\end{tikzpicture}
\caption{An illustration of the contradiction we reach in Theorem $1$. Points with unspecified color are grey, and the point that cannot be any color without forming a monochromatic $L$ is black.}
\label{Lproof}
\end{figure}

\subsection{Intervals}

We begin this section with a definition and a lemma concerning the minimum number of blue points on a given diagonal needed to force a contradiction.

\begin{definition}
    Let $n \geq 3$. Assume there is a 3-coloring of the $n \times n$ grid. Let $p$ be a point below the main diagonal. If $(1)$ $p$ is colored blue, and $(2)$ there are two red points on the main diagonal such that those points and $p$ form an $L$, then $p$ is \textit{forced by red to be blue}.
\end{definition}

\begin{lem}
   Let $n \geq 3$. Let 
    \[
    b = \left\lceil{\sqrt{2n-\frac{15}{4}}+ \frac{1}{2}}\right\rceil.
    \] If there are at least $b$ blue points on a diagonal which are forced by red to be blue, then there exists a monochromatic $L$ in the $n \times n$ grid.
\end{lem}

\begin{proof}
    Assume there is a $3$-coloring of the $n \times n$ grid. We show that $b$ blue points of the $n \times n$ grid fulfill these constraints. Assume there are $b$ blue points $p$ such that $p$ together with 2 red points on the diagonal form an $L$. Take any pair $\{p_1,p_2\}$ of these blue points. Let $q$ be the point such that $\{p_1,p_2,q\}$ form an $L$ with $q$ at the corner. Then note (a) if $q$ is blue there is a blue $L$, and (b) if $q$ is red there is a red $L$. Hence q must be green. We see there are $\binom{b}{2}$ green points forced by $p$ and $n-2$ subdiagonals on which these green points can occur, and so if $\frac{\binom{b}{2}}{n-2} > 1$, a monochromatic $L$ is forced. If $b$ is the least integer such that this inequality holds, we see the following:
    
    \begin{align*}
    \binom{b}{2} = \frac{b(b-1)}{2} & > n-2 \\
    b^2-b & > 2(n-2) \\
    b^2-b+\frac{1}{4} & > 2n-\frac{15}{4} \\
    b-\frac{1}{2} & > \sqrt{2n-\frac{15}{4}} \\
    b & = \left\lceil{\sqrt{2n-\frac{15}{4}}+ \frac{1}{2}}\right\rceil.
    \end{align*} 
\end{proof}

Having defined $b$, we now define intervals and discuss how we will use them to improve our bounds. An \textit{interval} $a_k$ on the main diagonal, where $k \in \mathbb Z$, $1 \leq k \leq \left\lceil{\frac{n}{3}}\right\rceil$, is the number of points between the $k$th and $(k+1)$th red points on the main diagonal. In general, if the $k$th and $(k+1)$th red points are separated by $c$ points, then $a_k = c$, and their corresponding point that cannot be colored red without forcing a monochromatic $L$ lies on $S_{c+1}$. For example, red points right next to each other have an interval of length $0$, red points separated by one non-red point have an interval of length $1$, and so on.

\begin{figure}[!h]
\centering
\begin{tikzpicture}[scale=0.97]
% a is the name of the node and A is the text inside the node/vertex  
    \node[fill=red] (1) at (0,6) {};  
    \node[fill=blue] (2) at (0,5) {};  
    \node[fill=gray] (3) at (0, 4) {};  
    \node[fill=blue] (4) at (0, 3) {}; 
    \node[fill=green] (5) at (0, 2) {}; 
    \node[fill=gray] (6) at (0, 1) {}; 
    \node[fill=green] (7) at (0, 0) {}; 
    \node[fill=red] (8) at (1, 5) {}; 
    \node[fill=gray] (9) at (1, 4) {}; 
    \node[fill=blue] (10) at (1, 3) {}; 
    \node[fill=blue] (11) at (1, 2) {}; 
    \node[fill=gray] (12) at (1, 1) {}; 
    \node[fill=green] (13) at (1, 0) {}; 
    \node[fill=gray] (14) at (2, 4) {}; 
    \node[fill=gray] (15) at (2, 3) {}; 
    \node[fill=gray] (16) at (2, 2) {}; 
    \node[fill=gray] (17) at (2, 1) {}; 
    \node[fill=gray] (18) at (2, 0) {}; 
    \node[fill=red] (19) at (3, 3) {}; 
    \node[fill=blue] (20) at (3, 2) {}; 
    \node[fill=gray] (21) at (3, 1) {}; 
    \node[fill=blue] (22) at (3, 0) {};
    \node[fill=red] (23) at (4, 2) {};
    \node[fill=gray] (24) at (4, 1) {};
    \node[fill=blue] (25) at (4, 0) {};
    \node[fill=gray] (26) at (5, 1) {};
    \node[fill=gray] (27) at (5, 0) {};
    \node[fill=red] (28) at (6, 0) {};
\end{tikzpicture}
\caption{An illustration of how certain intervals force points on corresponding diagonals to be either blue or green. For example, two adjacent red points force a blue/green point on the subdiagonal right below.}
\end{figure}
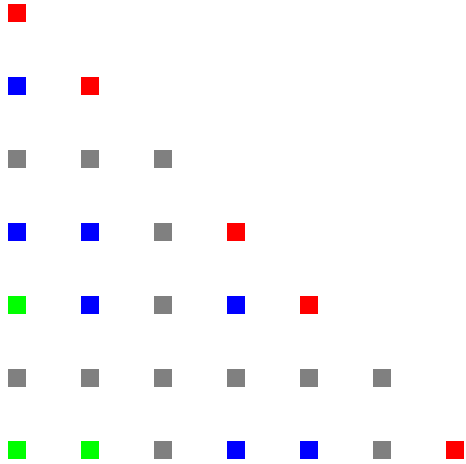

Now, consider our $n \times n$ integer grid. We cannot have more than $2(b-1)$ forced blue or green points in one diagonal, or else we must have either $b$ blue or $b$ green points, which forces a monochromatic $L$. Let $k_i$ be the number of pairs of consecutive red points on the main diagonal separated by $i$ points that are not red. Observe that $k_i \leq 2(b-1)$ for all possible $k_i$, or else $b$ blue points (we assume majority blue) will lie on $S_{i+1}$. The space $s$ taken up on the main diagonal by the red points and the intervals between consecutive red points is the following:
\[ s = \left\lceil{\frac{n}{3}}\right\rceil + 0k_0 + 1k_1 + \dots + (n-2)k_{n-2}. \]

In the above formula, $\ceil{\frac{n}{3}}$ represents the space taken up by our red points (we have $\ceil{\frac{n}{3}}$ of them and will prove this later), $0k_0$ represents the amount of space taken up by the intervals of length $0$ (as each of these intervals takes up $0$ space), $1k_1$ represents the amount of space taken up by the intervals of length $1$, and so on.

The way to get a lower bound on $s$ is to make these intervals as small as possible while still maintaining that there are not $b$ blue or $b$ green points on any subdiagonal created by these intervals. This means that $2(b-1)$ of these intervals will be length $0$, $2(b-1)$ will be length $1$, and so on until we run out of intervals. In other words, we take $k_0, k_1, \dots, k_{q-1} = 2(b-1)$, and $k_q = r$ where \[ q = \left\lfloor{\frac{\ceil{\frac{n}{3}}-1}{2(b-1)}}\right\rfloor, r = \left(\left\lceil{\frac{n}{3}}\right\rceil-1\right)-q \cdot 2(b-1) = \left\lceil{\frac{n}{3}}\right\rceil-1 \hspace{-0.1in} \pmod{2(b-1)}. \]

As such, we get the following lower bound for $s$ (which we call $s_{\min}$):
\begin{align*}
    s_{\min} & = \left\lceil{\frac{n}{3}}\right\rceil + (0k_0 + 1k_1 + \dots + (q-1)k_{q-1}) + qk_q \\
    & = \left\lceil{\frac{n}{3}}\right\rceil + 2(b-1)(0 + 1 + \dots + (q-1)) + qr \\
    & = \left\lceil{\frac{n}{3}}\right\rceil + 2(b-1)\binom{q}{2} + qr.
\end{align*}

To verify that $s_{\min}$ is minimized when we have as few red points on the main diagonal as possible, suppose there are $m > \left\lceil{\frac{n}{3}}\right\rceil$ points on the main diagonal. Then we have \[ q_m = \left\lfloor{\frac{m-1}{2(b-1)}}\right\rfloor, r_m = m-1 \hspace{-0.1in} \pmod{2(b-1)}, \] and so the minimum amount of space they take up (using the above argument) is 
\begin{align*}
    & m + (0k_0+1k_1+\dots+(q-1)k_{q-1}+qk_q+ \dots + qk_{q_m}) \\
    & > \left\lceil{\frac{n}{3}}\right\rceil + (0k_0 + 1k_1 + \dots + (q-1)k_{q-1}) + qk_q = s_{\min}.
\end{align*}

Thus, $s_{\min}$ occurs when we have as few red points on the main diagonal as possible. If $s_{\min} > n$, then the lower bound for the necessary space on the main diagonal is larger than the diagonal itself, meaning the red points cannot ``fit" on this diagonal without forcing a monochromatic $L$. The following theorem tells us which number this is.

\begin{theorem}
The value of $R_3(L)$ is at most $1804$.
\label{1804}\end{theorem}

\begin{proof}
Assume there is a $3$-coloring of the $n \times n$ grid. We wish to find the least $n \in \mathbb N$ such that $s_{\min} > n$. The above problem is an optimization problem, and so it is possible to solve it with a Python script. The program works in the following way: we start with a value of $n$ for which we can be sure this contradiction does not arise (say, $n=100$). We then set $s_{\min}$ and run a while loop while $s_{\min} \leq n$. For each iteration, we increase $n$ by $1$ and set all variables outlined in the proof ($b$, $s_{\min}$, $q$, and $r$) for the given $n$. This algorithm increases $n$ by $1$ until we have an $n$ such that $s_{\min} > n$. This $n$ is $1804$, and so the smallest grid that must have a monochromatic $L$ in any of its $3$-colorings is at most $1804 \times 1804$.
\end{proof}

Algorithm 1 on the page below contains the pseudocode for Theorem \ref{1804}.

\begin{algorithm}
\caption{Proving $R_3(L) \leq 1804$}
    \begin{algorithmic}
    \State{$n = 100$}
    \While{$s_{\min} \leq n$}
        \State{$n = n+1$}
        \State{$b = \left\lceil{\sqrt{2n-\frac{15}{4}}+ \frac{1}{2}}\right\rceil$}
        \State{$q = \left\lfloor{\frac{\ceil{\frac{n}{3}}-1}{2(b-1)}}\right\rfloor$}
        \State{$r = \left\lceil{\frac{n}{3}}\right\rceil-1 \pmod{2(b-1)}$}
        \State{$s_{\min} = \left\lceil{\frac{n}{3}}\right\rceil + 2(b-1)\binom{q}{2} + qr$}
    \EndWhile
\end{algorithmic}
\end{algorithm}

\subsection{Unaccounted Intervals}
\hspace{0.2in} One may notice that in the above proof, we only accounted for intervals between consecutive points. This is to make the proof simpler, but we can even improve on this by taking into account other intervals.

To give an idea of how this might be done, let us define the \textit{interval} $a_{j,k}$ as the number of points between the $j$th and $k$th red points as. From how we defined $a_j$, we note that $a_{j,k} = a_j$. Now, consider the space between the first and third red points. Between these, there exists the second red point and $a_1+a_2$ points that are not red. Thus, the interval between these points $a_{1,3}$ is of size $a_1+a_2+1$. To generalize, between points $j$ and $k$, there exist $k-j-1$ red points as well as intervals $a_j, a_{j+1}, \dots, a_{k-1}$. Thus, $$a_{j,k} = \sum_{i=j}^{k-1}a_i + k-j-1.$$

Let us extrapolate our property concerning intervals in Theorem \ref{1804} to these intervals $a_{j,k}$. Using $n$ and $b$ as defined in that proof, we say that for a given nonnegative integer $m$, there are at most $2(b-1)$ pairs $j,k$ where $1 \leq j, k \leq \left\lceil{\frac{n}{3}}\right\rceil$ such that $a_{j,k} = m$. In simpler terms, out of all intervals between red points on the main diagonal (not just consecutive points), there can be at most $2(b-1)$ intervals of a given size.

An immediate problem is how to sequence the intervals in a diagonal of given length. As an example, take $n=1803$, the largest number under our current bound. For this given $n$, we have $b = 61$, and so there can be at most $120$ intervals of any given size. In this instance, we have at least $\left\lceil{\frac{1803}{3}}\right\rceil = 601$ red points and thus at least $600$ intervals. If we make $120$ intervals each of length $0$, $1$, $2$, $3$, and $4$, then we have a total of $120 \cdot (0+1+2+3+4) + 601 = 1801$ minimum spaces taken up on the main diagonal. We wish to find a sequence of these intervals such that there are no more than $2(b-1)$ intervals of any given size.

We note that if $a_{j} = 0$, $a_{j+1} \leq 3$, then $a_{j,j+2} \leq 4$. Since we have exactly $2(b-1)$ intervals of length $0, 1, 2, 3,$ and $4$, this scenario would create another interval of length $4$, resulting in a contradiction. Thus, an interval of length $0$ can only have an interval of length $4$ on its right. By the same logic, an interval of length $0$ can only have an interval of length $4$ on its left hand side, so a $0$ can only be adjacent to a $4$. Using this line of reasoning, an interval of length $1$ can only be adjacent to an interval of length $3$ or $4$, and an interval of length $2$ can only be adjacent to an interval of length $2$, $3$, or $4$. Since we have the same number of $0$s and $4$s, if the sequence $0, 4, 0, 4, \dots, 0, 4$ exists, the first $0$ must be $a_1$ (or else the first interval of $0$ will necessarily border an interval less than length $4$ to its left). Similarly, if the subsequence $4, 0, 4, 0, \dots, 4, 0$ exists, the last interval of length $0$ must be $a_{600}$ in the main sequence.

Since all the $0$s ``pair up" with the $4$s, the $1$s can only be adjacent to $3$s or the $4$s at the ends of the subsequences listed above. Since $2$s can neighbor $2$s, $3$s, and $4$s, we can gather some sequences such that all intervals $a_{j, j+2}$ are greater than $4$:
\[
0, 4, 0, 4, \dots, 0, 4, 1, 3, 1, 3, \dots, 1, 3, 2, 2, \dots, 2, 3, 1, 3, 1, \dots, 3, 1.
\]
\[
1, 3, 1, 3, \dots, 1, 3, 2, 2, \dots, 2, 3, 1, 3, 1, \dots, 3, 1, 4, 0, 4, 0, \dots.
\]
\[
0, 4, 0, 4, \dots 1, 3, 1, 3, \dots, 1, 3, 2, 2, \dots, 2, 3, 1, 3, 1, \dots, 3, 1, 4, 0, 4, 0, \dots.
\]

Note that if we group all adjacent intervals into pairs such that $\{a_1, a_2\}$, $\{a_3, a_4\}$ and so on are pairs, then nearly all pairs add up to $4$. Recall that $a_{j,j+2} = a_j+a_{j+1}+1$, and so from this we generate around $300$ intervals of size $5$. This is a contradiction since we can have at most $120$ intervals of any given size. From this we gather that we cannot use only the intervals $0, 1, 2, 3,$ and $4$: we must change some of these to a different size. 

While we do not work further with this sequencing of intervals in this paper, the idea to include intervals between nonconsecutive red points motivates another lowering of our upper bound as detailed in the proof below.

\begin{theorem}
The value of $R_3(L)$ is at most $1573$.
\label{1573}\end{theorem}

\begin{proof}
Assume there is a $3$-coloring of the $n \times n$ grid. We show that a $1573 \times 1573$ integer grid must contain a monochromatic $L$ using intervals between non-adjacent red points on our main diagonal. We can see that the minimum space taken up by the red points on the main diagonal of length $n$ is equivalent to $n^*+a_1+a_2+a_3+\dots+a_{n^*-1}$, where $n^* = \left\lceil{\frac{n}{3}}\right\rceil$. A space calculated using similar methodology to the one in the proof of Theorem \ref{1804} can be measured by taking the intervals between red points $1$ and $3$, $3$ and $5$, and so on until we run out of points. That is, for $t \in \mathbb N$ such that $t$ is the greatest odd number such that $t \leq n^*$, the sum of the intervals $a_{1,3}, a_{3,5}, \dots a_{t-2, t}$ plus the unaccounted points between these intervals will be the same space taken up by the sum of $a_1, a_2, \dots, a_{t-1}$, as well as all the unaccounted points in this interval. In mathematical terms, we have the following inequality:
\begin{align*}
    a_{1,3}+a_{3,5}+\dots+a_{t-2,t}+\left\lfloor{\frac{n^*-1}{2}}\right\rfloor + 1 & = a_1+a_2+\dots+a_{t-1}+t \\
    & \leq a_1+a_2+\dots+a_{n^*-1}+n^* = s \leq n. 
\end{align*}
Combining this equation with our first gives us the following: $$a_1+a_2+\dots+a_{n^*-1}+a_{1,3}+a_{3,5}+\dots+a_{t-2,t}+n^*+\left\lfloor{\frac{n^*-1}{2}}\right\rfloor+1 \leq 2n.$$

In this formula, we are counting both the space taken up by the red points and the intervals between them as well as every other red point and the space between these points. Let us call this combined space $s^*$. As before, at most $2(b-1)$ of these intervals can have length $m$ where $m$ is a nonnegative integer. Assume the intervals in the above equation are as small as possible in order to minimize space (we will call the minimum space $s^*_{\min}$). Let $q, r \in \mathbb Z$ such that $q$ is the quotient upon dividing the number of intervals $\left((n^*-1)+\left\lfloor{\frac{n^*-1}{2}}\right\rfloor\right)$ by $2(b-1)$ and $r$ is the remainder. This gives us the following: 
\begin{align*}
    s^*_{\min} & = n^* + \left\lfloor{\frac{n^*-1}{2}}\right\rfloor + 1 + 2(b-1)(0 + 1 + \dots + (q-1)) + qr \\
    & = n^* + \left\lfloor{\frac{n^*-1}{2}}\right\rfloor + 1 + 2(b-1)\binom{q}{2} + qr.
\end{align*}
We use a Python script to generate the first number $n$ such that this condition does not hold. Similar to the last program, we iterate through each $n \in \mathbb N$ in increasing order, setting $b, q, r,$ and $s^*_{\min}$ for the current $n$, and checking to see whether $s^*_{\min} > 2n$. This program gives us $n=1573$.
\end{proof}

The pseudocode for Theorem \ref{1573} is included as Algorithm 2 below.

\begin{algorithm}
\caption{Proving $R_3(L) \leq 1573$}
    \begin{algorithmic}
    \State{$n \gets 100$}
    \While{$s_{\min} \leq 2n$}
        \State{$n \gets n+1$}
        \State{$n^* \gets \left\lceil{\frac{n}{3}}\right\rceil$}
        \State{intvls $\gets (n^*-1)+\left\lfloor{\frac{n^*-1}{2}}\right\rfloor$}
        \State{$b \gets \left\lceil{\sqrt{2n-\frac{15}{4}}+ \frac{1}{2}}\right\rceil$}
        \State{$q \gets \left\lfloor{\frac{\text{intvls}}{2(b-1)}}\right\rfloor$}
        \State{$r \gets (\text{intvls}) \pmod{2(b-1)}$}
        \State{$s_{\min} = \text{intvls} + 2 + 2(b-1)\binom{q}{2} + qr$}
    \EndWhile
    \State{\Return{n}}
\end{algorithmic}
\end{algorithm}

\subsection{Further Improvements}
In this section, we discuss further major improvements made to the upper bound of $R_3(L)$.

\begin{theorem}
    The value of $R_3(L)$ is at most $772$.
\label{772}\end{theorem}

\begin{proof}
    Recall that if there are $2b-1$ pairs of red points, each some fixed distance $k$ apart, then there will be $2b-1$ points forced to be either blue or green in subdiagonal $S_k$. This implies there will be $b$ points in $S_k$ of some fixed color, say blue, thus forcing 2 green points in the same subdiagonal which in turn force a monochromatic L.

    For $c \leq n^*-1$, there are $n^*-c$ intervals of the form $a_{j, j+c}$. We can partition these intervals into the following sums:
    \begin{align*}
        & a_{1, 1+c} + a_{1+c, 1+2c} + \dots \\
        & a_{2, 2+c} + a_{2+c, 2+2c} + \dots \\
        & \hspace{0.6in} \vdots \\
        & a_{c, 2c} + a_{2c, 3c} + \dots.
    \end{align*}
    Since the intervals in each sum are consecutive on the main diagonal, each sum plus all the included red points can be at most $n$. Each red point is included once across all sums, so we get the following formula: $$\left(\sum_{i=1}^{n^*-c} a_{i, i+c}\right) + n^* \leq cn.$$ Considering the space taken up by all intervals with length \textit{at most} c, we get the following formula: $$\left(\sum_{j=1}^{c}\sum_{i=1}^{n^*-j} a_{i, i+j}\right) + cn^* \leq n \cdot \frac{c(c+1)}{2}.$$ Moreover, $k_i \leq 2b-1$ for every $i$. For $n=772$ and $c=12$, both of these conditions cannot be true at the same time. Thus, the $n \times n$ grid must contain a monochromatic $L$.
\end{proof}

Before we discuss the next improvement, we introduce the concept of Golomb rulers to those not acquainted. A \textit{Golomb ruler} is a set of integers such that no two pairs of integers are the same distance apart. The number of integers in a Golomb ruler is its \textit{order}, and the distance between the largest and smallest integers in a Golomb ruler is its \textit{length}. A Golomb ruler is \textit{optimal} if for all Golomb rulers with the same order, there are none with smaller length.

\begin{theorem}
    The value of $R_3(L)$ is at most $493$.
\end{theorem}

\begin{proof}
    Consider the points forced by red to be blue. In any given subdiagonal, if any two pairs of these blue points are the same distance apart, then two green points forced by blue lie on some subdiagonal below. This forces a monochromatic L. So by definition, these blue points must form a Golomb ruler. We can thus lower our bound $b$ for the number of these blue points allowed to $b_k$, the largest order possible in a Golomb ruler of length $n-k-1$, or the largest number of forced blue points in $S_k$ such that no two pairs are the same distance apart. Applying $b_k$ instead of $b$ to Theorem \ref{772} gives us a contradiction when $n = 493$ and $c = 12$.
\end{proof}

The pseudocode for this algorithm is written in Algorithm 3 on the page below.
\begin{algorithm}
    \caption{Proving $R_3(L) \leq 493$}
    \begin{algorithmic}
        \State{Golomb $\gets$ \{0, 1, 2, 4, 7, 12, 18, 26, 35, 45, 56, 73, 86, 107, 128, 152, 178, 200, 217, 247, 284, 334, 357, 373, 426, 481, 493, 554, 586\}} \Comment{1 + the length of optimal Golomb rulers for orders $0, 1, \dots, 28$}
        \State{blue\_array} $\gets$ int[586]
        \State{$i, j \gets 0$}
        \While{$i <$ Golomb[length(Golomb)-1]}
            \State{blue\_array[$i$] $\gets j$}
            \If{Golomb[$j$] = $i$}
                \State{$j \gets j+1$}
            \EndIf
            \State{$i \gets i+1$}
        \EndWhile \Comment{Set blue\_array[$k$] $= b_{n-k}$}
        \State{$c \gets 5$}
        \While{$c < 20$} \Comment{Testing different $c$ values to find which one gives the lowest bound}
            \State{sum, space $\gets 0$}
            \State{$n \gets 100$}
            \While{sum $<$ space} \Comment{Find least $n$ such that we reach a contradiction}
                \State{$n \gets n+1$}
                \State{space $\gets \frac{nc(c+1)}{2}$}
                \State{$n^* \gets \ceil{\frac{n}{3}}$}
                \State{intvls $\gets cn^* - \frac{c(c+1)}{2}$}
                \State{sum $\gets \frac{c(c+1)(c+2)}{6}+\frac{c(c-1)(c+1)}{6}+$ints}
                \State{$k \gets 1$}
                \While{intvls $>$ 0} \Comment{Setting intervals to smallest lengths possible}
                    \If{2(blue\_array[$n-k$]) $>$ intvls}
                        \State{sum $\gets$ sum + intvls$(k-1)$}
                        \State{intvls=0}
                    \Else
                        \State{sum $\gets$ sum + 2(blue\_array[$n-k$])(k-1)}
                        \State{intvls $\gets$ intvls-2(blue\_array[$n-k$])}
                    \EndIf
                \EndWhile
            \EndWhile
        \EndWhile
    \end{algorithmic}
\end{algorithm}
\section{The Lower Bound/On SAT Solvers}
In the paper of Canacki et. al., the authors used SAT solvers to find a $20 \times 20$ grid with no monochromatic $L$s \cite{canackisat}. We used a parallel SAT solver to find many other such grids. Below we will describe the process and results.

For SAT solvers, we use a boolean format called DIMACS CNF. We start with different variables that are represented by characters (in our case, we use the integers starting from $1$). A variable takes a negative value if it is false and positive if it is true (for example, if variable $1$ is false, we represent this as $-1$). The entire boolean statement is represented as a collection of clauses (a collection of variables or negations of variables conjoined by OR), with each clause conjoined by AND. In CNF format, each clause is represented as a list of each variable, and all clauses are separated by a line break. For example, if we wanted to write the boolean statement ``1 is false or 2 is true, and 2 is false", it would look the following way: 
\begin{align*}
    & -1 \hspace{0.13in} 2 \\
& -2 
\end{align*}

To start, we assigned to each possible state of each point a value. The specific value assigned to color $k$ in row $r$ and column $c$ in an $n \times n$ grid was $(k+1)+3(c-1)+3n(r-1)$. For example, color $1$ in the point in the 2nd row and first column of a $20 \times 20$ grid would take a value of $2+3(0)+3(20)(1) = 62$.

For our particular instance, if a particular color in a particular row in a particular column exists in a clause, its numerical representation is included; if it does not, its negation is included. A program used to find Sudoku solutions was modified to generate the CNF statement for an $n \times n$ grid. There are $3$ overarching parameters we must check; that every point takes some color, every point takes only one color, and that no monochromatic $L$ is formed. This input is plugged into the parallel SAT solver, and then afterwards is processed by another piece of code that takes the solution from the SAT solver and gives the colored $n \times n$ grid. The grid generated for $n = 20$ with no monochromatic $L$ is the following:
\begin{figure}[!h]
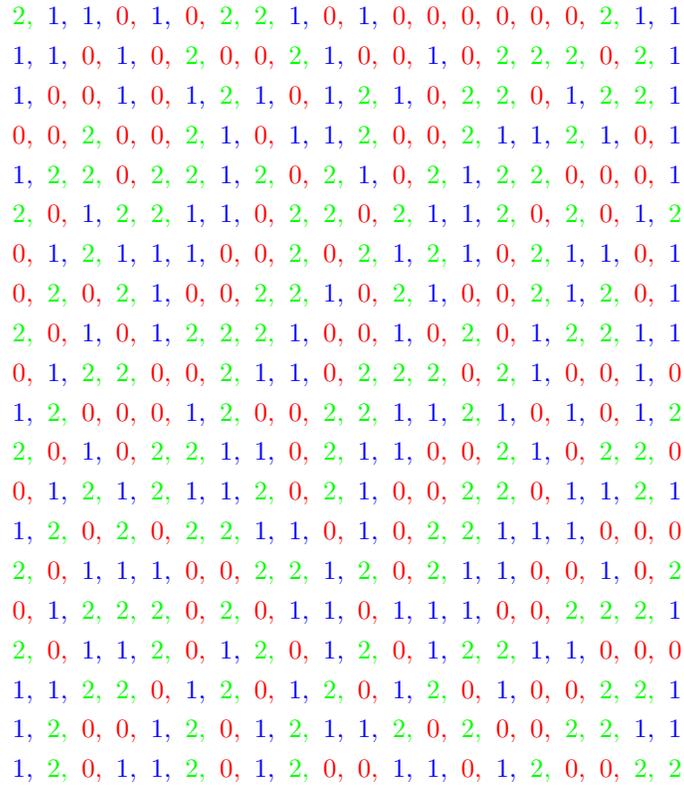

\begin{align*}
\color{green} 2,\hspace{0.05in} \color{blue} 1,\hspace{0.05in} \color{blue} 1,\hspace{0.05in} \color{red} 0,\hspace{0.05in} \color{blue} 1,\hspace{0.05in} \color{red} 0,\hspace{0.05in} \color{green} 2,\hspace{0.05in} \color{green} 2,\hspace{0.05in} \color{blue} 1,\hspace{0.05in} \color{red} 0,\hspace{0.05in} \color{blue} 1,\hspace{0.05in} \color{red} 0,\hspace{0.05in} \color{red} 0,\hspace{0.05in} \color{red} 0,\hspace{0.05in} \color{red} 0,\hspace{0.05in} \color{red} 0,\hspace{0.05in} \color{red} 0,\hspace{0.05in} \color{green} 2,\hspace{0.05in} \color{blue} 1,\hspace{0.05in} \color{blue} 1 \\
\color{blue} 1,\hspace{0.05in} \color{blue} 1,\hspace{0.05in} \color{red} 0,\hspace{0.05in} \color{blue} 1,\hspace{0.05in} \color{red} 0,\hspace{0.05in} \color{green} 2,\hspace{0.05in} \color{red} 0,\hspace{0.05in} \color{red} 0,\hspace{0.05in} \color{green} 2,\hspace{0.05in} \color{blue} 1,\hspace{0.05in} \color{red} 0,\hspace{0.05in} \color{red} 0,\hspace{0.05in} \color{blue} 1,\hspace{0.05in} \color{red} 0,\hspace{0.05in} \color{green} 2,\hspace{0.05in} \color{green} 2,\hspace{0.05in} \color{green} 2,\hspace{0.05in} \color{red} 0,\hspace{0.05in} \color{green} 2,\hspace{0.05in} \color{blue} 1 \\
\color{blue} 1,\hspace{0.05in} \color{red} 0,\hspace{0.05in} \color{red} 0,\hspace{0.05in} \color{blue} 1,\hspace{0.05in} \color{red} 0,\hspace{0.05in} \color{blue} 1,\hspace{0.05in} \color{green} 2,\hspace{0.05in} \color{blue} 1,\hspace{0.05in} \color{red} 0,\hspace{0.05in} \color{blue} 1,\hspace{0.05in} \color{green} 2,\hspace{0.05in} \color{blue} 1,\hspace{0.05in} \color{red} 0,\hspace{0.05in} \color{green} 2,\hspace{0.05in} \color{green} 2,\hspace{0.05in} \color{red} 0,\hspace{0.05in} \color{blue} 1,\hspace{0.05in} \color{green} 2,\hspace{0.05in} \color{green} 2,\hspace{0.05in} \color{blue} 1 \\
\color{red} 0,\hspace{0.05in} \color{red} 0,\hspace{0.05in} \color{green} 2,\hspace{0.05in} \color{red} 0,\hspace{0.05in} \color{red} 0,\hspace{0.05in} \color{green} 2,\hspace{0.05in} \color{blue} 1,\hspace{0.05in} \color{red} 0,\hspace{0.05in} \color{blue} 1,\hspace{0.05in} \color{blue} 1,\hspace{0.05in} \color{green} 2,\hspace{0.05in} \color{red} 0,\hspace{0.05in} \color{red} 0,\hspace{0.05in} \color{green} 2,\hspace{0.05in} \color{blue} 1,\hspace{0.05in} \color{blue} 1,\hspace{0.05in} \color{green} 2,\hspace{0.05in} \color{blue} 1,\hspace{0.05in} \color{red} 0,\hspace{0.05in} \color{blue} 1 \\
\color{blue} 1,\hspace{0.05in} \color{green} 2,\hspace{0.05in} \color{green} 2,\hspace{0.05in} \color{red} 0,\hspace{0.05in} \color{green} 2,\hspace{0.05in} \color{green} 2,\hspace{0.05in} \color{blue} 1,\hspace{0.05in} \color{green} 2,\hspace{0.05in} \color{red} 0,\hspace{0.05in} \color{green} 2,\hspace{0.05in} \color{blue} 1,\hspace{0.05in} \color{red} 0,\hspace{0.05in} \color{green} 2,\hspace{0.05in} \color{blue} 1,\hspace{0.05in} \color{green} 2,\hspace{0.05in} \color{green} 2,\hspace{0.05in} \color{red} 0,\hspace{0.05in} \color{red} 0,\hspace{0.05in} \color{red} 0,\hspace{0.05in} \color{blue} 1 \\
\color{green} 2,\hspace{0.05in} \color{red} 0,\hspace{0.05in} \color{blue} 1,\hspace{0.05in} \color{green} 2,\hspace{0.05in} \color{green} 2,\hspace{0.05in} \color{blue} 1,\hspace{0.05in} \color{blue} 1,\hspace{0.05in} \color{red} 0,\hspace{0.05in} \color{green} 2,\hspace{0.05in} \color{green} 2,\hspace{0.05in} \color{red} 0,\hspace{0.05in} \color{green} 2,\hspace{0.05in} \color{blue} 1,\hspace{0.05in} \color{blue} 1,\hspace{0.05in} \color{green} 2,\hspace{0.05in} \color{red} 0,\hspace{0.05in} \color{green} 2,\hspace{0.05in} \color{red} 0,\hspace{0.05in} \color{blue} 1,\hspace{0.05in} \color{green} 2 \\
\color{red} 0,\hspace{0.05in} \color{blue} 1,\hspace{0.05in} \color{green} 2,\hspace{0.05in} \color{blue} 1,\hspace{0.05in} \color{blue} 1,\hspace{0.05in} \color{blue} 1,\hspace{0.05in} \color{red} 0,\hspace{0.05in} \color{red} 0,\hspace{0.05in} \color{green} 2,\hspace{0.05in} \color{red} 0,\hspace{0.05in} \color{green} 2,\hspace{0.05in} \color{blue} 1,\hspace{0.05in} \color{green} 2,\hspace{0.05in} \color{blue} 1,\hspace{0.05in} \color{red} 0,\hspace{0.05in} \color{green} 2,\hspace{0.05in} \color{blue} 1,\hspace{0.05in} \color{blue} 1,\hspace{0.05in} \color{red} 0,\hspace{0.05in} \color{blue} 1 \\
\color{red} 0,\hspace{0.05in} \color{green} 2,\hspace{0.05in} \color{red} 0,\hspace{0.05in} \color{green} 2,\hspace{0.05in} \color{blue} 1,\hspace{0.05in} \color{red} 0,\hspace{0.05in} \color{red} 0,\hspace{0.05in} \color{green} 2,\hspace{0.05in} \color{green} 2,\hspace{0.05in} \color{blue} 1,\hspace{0.05in} \color{red} 0,\hspace{0.05in} \color{green} 2,\hspace{0.05in} \color{blue} 1,\hspace{0.05in} \color{red} 0,\hspace{0.05in} \color{red} 0,\hspace{0.05in} \color{green} 2,\hspace{0.05in} \color{blue} 1,\hspace{0.05in} \color{green} 2,\hspace{0.05in} \color{red} 0,\hspace{0.05in} \color{blue} 1 \\
\color{green} 2,\hspace{0.05in} \color{red} 0,\hspace{0.05in} \color{blue} 1,\hspace{0.05in} \color{red} 0,\hspace{0.05in} \color{blue} 1,\hspace{0.05in} \color{green} 2,\hspace{0.05in} \color{green} 2,\hspace{0.05in} \color{green} 2,\hspace{0.05in} \color{blue} 1,\hspace{0.05in} \color{red} 0,\hspace{0.05in} \color{red} 0,\hspace{0.05in} \color{blue} 1,\hspace{0.05in} \color{red} 0,\hspace{0.05in} \color{green} 2,\hspace{0.05in} \color{red} 0,\hspace{0.05in} \color{blue} 1,\hspace{0.05in} \color{green} 2,\hspace{0.05in} \color{green} 2,\hspace{0.05in} \color{blue} 1,\hspace{0.05in} \color{blue} 1 \\
\color{red} 0,\hspace{0.05in} \color{blue} 1,\hspace{0.05in} \color{green} 2,\hspace{0.05in} \color{green} 2,\hspace{0.05in} \color{red} 0,\hspace{0.05in} \color{red} 0,\hspace{0.05in} \color{green} 2,\hspace{0.05in} \color{blue} 1,\hspace{0.05in} \color{blue} 1,\hspace{0.05in} \color{red} 0,\hspace{0.05in} \color{green} 2,\hspace{0.05in} \color{green} 2,\hspace{0.05in} \color{green} 2,\hspace{0.05in} \color{red} 0,\hspace{0.05in} \color{green} 2,\hspace{0.05in} \color{blue} 1,\hspace{0.05in} \color{red} 0,\hspace{0.05in} \color{red} 0,\hspace{0.05in} \color{blue} 1,\hspace{0.05in} \color{red} 0 \\
\color{blue} 1,\hspace{0.05in} \color{green} 2,\hspace{0.05in} \color{red} 0,\hspace{0.05in} \color{red} 0,\hspace{0.05in} \color{red} 0,\hspace{0.05in} \color{blue} 1,\hspace{0.05in} \color{green} 2,\hspace{0.05in} \color{red} 0,\hspace{0.05in} \color{red} 0,\hspace{0.05in} \color{green} 2,\hspace{0.05in} \color{green} 2,\hspace{0.05in} \color{blue} 1,\hspace{0.05in} \color{blue} 1,\hspace{0.05in} \color{green} 2,\hspace{0.05in} \color{blue} 1,\hspace{0.05in} \color{red} 0,\hspace{0.05in} \color{blue} 1,\hspace{0.05in} \color{red} 0,\hspace{0.05in} \color{blue} 1,\hspace{0.05in} \color{green} 2 \\
\color{green} 2,\hspace{0.05in} \color{red} 0,\hspace{0.05in} \color{blue} 1,\hspace{0.05in} \color{red} 0,\hspace{0.05in} \color{green} 2,\hspace{0.05in} \color{green} 2,\hspace{0.05in} \color{blue} 1,\hspace{0.05in} \color{blue} 1,\hspace{0.05in} \color{red} 0,\hspace{0.05in} \color{green} 2,\hspace{0.05in} \color{blue} 1,\hspace{0.05in} \color{blue} 1,\hspace{0.05in} \color{red} 0,\hspace{0.05in} \color{red} 0,\hspace{0.05in} \color{green} 2,\hspace{0.05in} \color{blue} 1,\hspace{0.05in} \color{red} 0,\hspace{0.05in} \color{green} 2,\hspace{0.05in} \color{green} 2,\hspace{0.05in} \color{red} 0 \\
\color{red} 0,\hspace{0.05in} \color{blue} 1,\hspace{0.05in} \color{green} 2,\hspace{0.05in} \color{blue} 1,\hspace{0.05in} \color{green} 2,\hspace{0.05in} \color{blue} 1,\hspace{0.05in} \color{blue} 1,\hspace{0.05in} \color{green} 2,\hspace{0.05in} \color{red} 0,\hspace{0.05in} \color{green} 2,\hspace{0.05in} \color{blue} 1,\hspace{0.05in} \color{red} 0,\hspace{0.05in} \color{red} 0,\hspace{0.05in} \color{green} 2,\hspace{0.05in} \color{green} 2,\hspace{0.05in} \color{red} 0,\hspace{0.05in} \color{blue} 1,\hspace{0.05in} \color{blue} 1,\hspace{0.05in} \color{green} 2,\hspace{0.05in} \color{blue} 1 \\
\color{blue} 1,\hspace{0.05in} \color{green} 2,\hspace{0.05in} \color{red} 0,\hspace{0.05in} \color{green} 2,\hspace{0.05in} \color{red} 0,\hspace{0.05in} \color{green} 2,\hspace{0.05in} \color{green} 2,\hspace{0.05in} \color{blue} 1,\hspace{0.05in} \color{blue} 1,\hspace{0.05in} \color{red} 0,\hspace{0.05in} \color{blue} 1,\hspace{0.05in} \color{red} 0,\hspace{0.05in} \color{green} 2,\hspace{0.05in} \color{green} 2,\hspace{0.05in} \color{blue} 1,\hspace{0.05in} \color{blue} 1,\hspace{0.05in} \color{blue} 1,\hspace{0.05in} \color{red} 0,\hspace{0.05in} \color{red} 0,\hspace{0.05in} \color{red} 0 \\
\color{green} 2,\hspace{0.05in} \color{red} 0,\hspace{0.05in} \color{blue} 1,\hspace{0.05in} \color{blue} 1,\hspace{0.05in} \color{blue} 1,\hspace{0.05in} \color{red} 0,\hspace{0.05in} \color{red} 0,\hspace{0.05in} \color{green} 2,\hspace{0.05in} \color{green} 2,\hspace{0.05in} \color{blue} 1,\hspace{0.05in} \color{green} 2,\hspace{0.05in} \color{red} 0,\hspace{0.05in} \color{green} 2,\hspace{0.05in} \color{blue} 1,\hspace{0.05in} \color{blue} 1,\hspace{0.05in} \color{red} 0,\hspace{0.05in} \color{red} 0,\hspace{0.05in} \color{blue} 1,\hspace{0.05in} \color{red} 0,\hspace{0.05in} \color{green} 2 \\
\color{red} 0,\hspace{0.05in} \color{blue} 1,\hspace{0.05in} \color{green} 2,\hspace{0.05in} \color{green} 2,\hspace{0.05in} \color{green} 2,\hspace{0.05in} \color{red} 0,\hspace{0.05in} \color{green} 2,\hspace{0.05in} \color{red} 0,\hspace{0.05in} \color{blue} 1,\hspace{0.05in} \color{blue} 1,\hspace{0.05in} \color{red} 0,\hspace{0.05in} \color{blue} 1,\hspace{0.05in} \color{blue} 1,\hspace{0.05in} \color{blue} 1,\hspace{0.05in} \color{red} 0,\hspace{0.05in} \color{red} 0,\hspace{0.05in} \color{green} 2,\hspace{0.05in} \color{green} 2,\hspace{0.05in} \color{green} 2,\hspace{0.05in} \color{blue} 1 \\
\color{green} 2,\hspace{0.05in} \color{red} 0,\hspace{0.05in} \color{blue} 1,\hspace{0.05in} \color{blue} 1,\hspace{0.05in} \color{green} 2,\hspace{0.05in} \color{red} 0,\hspace{0.05in} \color{blue} 1,\hspace{0.05in} \color{green} 2,\hspace{0.05in} \color{red} 0,\hspace{0.05in} \color{blue} 1,\hspace{0.05in} \color{green} 2,\hspace{0.05in} \color{red} 0,\hspace{0.05in} \color{blue} 1,\hspace{0.05in} \color{green} 2,\hspace{0.05in} \color{green} 2,\hspace{0.05in} \color{blue} 1,\hspace{0.05in} \color{blue} 1,\hspace{0.05in} \color{red} 0,\hspace{0.05in} \color{red} 0,\hspace{0.05in} \color{red} 0 \\
\color{blue} 1,\hspace{0.05in} \color{blue} 1,\hspace{0.05in} \color{green} \color{green} 2,\hspace{0.05in} \color{green} \color{green} 2,\hspace{0.05in} \color{red} 0,\hspace{0.05in} \color{blue} 1,\hspace{0.05in} \color{green} \color{green} 2,\hspace{0.05in} \color{red} 0,\hspace{0.05in} \color{blue} 1,\hspace{0.05in} \color{green} \color{green} 2,\hspace{0.05in} \color{red} 0,\hspace{0.05in} \color{blue} 1,\hspace{0.05in} \color{green} \color{green} 2,\hspace{0.05in} \color{red} 0,\hspace{0.05in} \color{blue} 1,\hspace{0.05in} \color{red} 0,\hspace{0.05in} \color{red} 0,\hspace{0.05in} \color{green} \color{green} 2,\hspace{0.05in} \color{green} \color{green} 2,\hspace{0.05in} \color{blue} 1 \\
\color{blue} 1,\hspace{0.05in} \color{green} \color{green} 2,\hspace{0.05in} \color{red} 0,\hspace{0.05in} \color{red} 0,\hspace{0.05in} \color{blue} 1,\hspace{0.05in} \color{green} \color{green} 2,\hspace{0.05in} \color{red} 0,\hspace{0.05in} \color{blue} 1,\hspace{0.05in} \color{green} \color{green} 2,\hspace{0.05in} \color{blue} 1,\hspace{0.05in} \color{blue} 1,\hspace{0.05in} \color{green} \color{green} 2,\hspace{0.05in} \color{red} 0,\hspace{0.05in} \color{green} \color{green} 2,\hspace{0.05in} \color{red} 0,\hspace{0.05in} \color{red} 0,\hspace{0.05in} \color{green} \color{green} 2,\hspace{0.05in} \color{green} \color{green} 2,\hspace{0.05in} \color{blue} 1,\hspace{0.05in} \color{blue} 1 \\
\color{blue} 1,\hspace{0.05in} \color{green} \color{green} 2,\hspace{0.05in} \color{red} 0,\hspace{0.05in} \color{blue} 1,\hspace{0.05in} \color{blue} 1,\hspace{0.05in} \color{green} \color{green} 2,\hspace{0.05in} \color{red} 0,\hspace{0.05in} \color{blue} 1,\hspace{0.05in} \color{green} \color{green} 2,\hspace{0.05in} \color{red} 0,\hspace{0.05in} \color{red} 0,\hspace{0.05in} \color{blue} 1,\hspace{0.05in} \color{blue} 1,\hspace{0.05in} \color{red} 0,\hspace{0.05in} \color{blue} 1,\hspace{0.05in} \color{green} \color{green} 2,\hspace{0.05in} \color{red} 0,\hspace{0.05in} \color{red} 0,\hspace{0.05in} \color{green} \color{green} 2,\hspace{0.05in} \color{green} \color{green} 2 
\end{align*}
\caption{A satisfiable $20 \times 20$ grid as found by CryptoMiniSAT. If the integer in the $i$th row and $j$th column is $c$, then the color of $(i, j)$ in our grid is color $c$. In this grid (as well as all subsequent grids in this paper), the $0$s are colored red, the $1$s are colored blue, and the $2$s are colored green.}
\end{figure}
 \begin{theorem}
     The number of CNF statements for checking whether a $3-$colored $n \times n$ grid has no monochromatic $L$ is $\frac{n(n-1)(2n-1)}{2} + 4n^2 = O(n^3)$.
 \end{theorem}

 \begin{proof}
     The CNF statements for checking this condition can be broken down into two categories: making sure each point has exactly one color, and making sure there are no monochromatic $L$s. For the first condition, let us look at an arbitrary point in row $r$ and column $c$. We are able to ensure that this point has $1$ color with $4$ CNF statements: we make sure that it has at least one color $$x_{r,c,0} \vee x_{r,c,1} \vee x_{r,c,2}$$ and that it has no more than one color 
     \begin{align*}
         -x_{r,c,0} & \vee -x_{r,c,1} \\
         -x_{r,c,0} & \vee -x_{r,c,2} \\
         -x_{r,c,1} & \vee -x_{r,c,2}. 
     \end{align*}

     In the statement above, we make sure that no combination of two colors exist at the same time. We have this many statements for $n^2$ points, which gives us a total of $4n^2$ statements for the first case.

     Now, let us look at the number of statements for the number of monochromatic $L$s. We can organize the number of $L$s by which point in the $L$ is the upper point. Consider the point $(r, c)$. We established that an $L$ is any valid collection of points $(r, c), (r+t, c), (r+t, c+t)$ for some $t \in \mathbb N$. These points are all valid as long as $r+t, c+t$ are not greater than the dimensions of this grid; that is, $r+t, c+t \leq n$ (and since $t$ is positive, $r, c \leq n-1$). So the number of $L$s whose upper point is $(r, c)$ is equivalent to $\min{n-1-r, n-1-c}$. Let us label each vertex of the $n \times n$ grid with how many $L$s stem from it. \\

    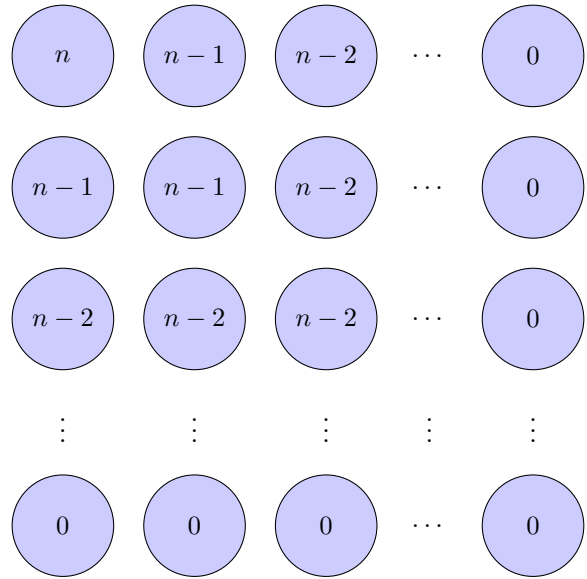
\begin{figure}[!h]
    \centering
    \begin{tikzpicture}[state/.style={circle, draw, minimum size=1.35cm, fill=blue!20}]
    % a is the name of the node and A is the text inside the node/vertex  
        \node[state] (1) at (0, 2.75) {$n-2$};  
        \node[state] (2) at (0, 4.5) {$n-1$};  
        \node[state] (3) at (0, 6.25) {$n$}; 
        \node[state] (4) at (1.75, 2.75) {$n-2$};  
        \node[state] (5) at (1.75, 4.5) {$n-1$};  
        \node[state] (6) at (1.75, 6.25) {$n-1$}; 
        \node[state] (7) at (3.5, 2.75) {$n-2$};  
        \node[state] (8) at (3.5, 4.5) {$n-2$};  
        \node[state] (9) at (3.5, 6.25) {$n-2$};
        \node[state] (10) at (0, 0) {$0$};
        \node[state] (11) at (1.75, 0) {$0$};
        \node[state] (12) at (3.5, 0) {$0$};
        \node[state] (13) at (6.25, 0) {$0$};
        \node[state] (14) at (6.25, 2.75) {$0$};
        \node[state] (15) at (6.25, 4.5) {$0$};
        \node[state] (16) at (6.25, 6.25) {$0$};
        \path (9) -- node[auto=false]{\ldots} (16);
        \path (8) -- node[auto=false]{\ldots} (15);
        \path (7) -- node[auto=false]{\ldots} (14);
        \path (1) -- node[auto=false]{\vdots} (10);
        \path (4) -- node[auto=false]{\vdots} (11);
        \path (7) -- node[auto=false]{\vdots} (12);
        \path (7) -- node[auto=false]{\vdots} (13);
        \path (12) -- node[auto=false]{\ldots} (13);
        \path (14) -- node[auto=false]{\vdots} (13);
  
    % Bidirected edge  
    %\path[bidirected] (a) edge[bend left=60] (b); % this is the basic command in this code. It is used to draw the curved edge with a certain angle. You can change the angle according to the requirements.  
      
    %\path (a) edge (c);  
    % \path[bidirected] (a) edge[bend right=60] (c);  
    % \draw (b) -- (c);  
\end{tikzpicture}
\caption{A representation of an $n \times n$ grid, where each point $p$ is labeled with the number of $L$s that can be formed with $p$ as the uppermost point.}
\end{figure}
Note that the grid with the points labeled as above can be formed by setting every point to have value $0$, then adding $1$ to the upper $n-1 \times n-1$ grid, then adding $1$ to the upper $n-2 \times n-2$ grid, and so on until we add $1$ to the top left point. This gives us a total of $(n-1)^2 + (n-2)^2 + \dots + 1$ $L$s, which is equivalent to $\frac{(n-1)((n-1)+1)(2(n-1)+1)}{6}$ $L$s using our formula for the sum of the first $n$ square numbers. For a given point, there are three CNF statements to check that there is no monochromatic $L$ starting at that point, one for each color. This gives us that the total number of CNF statements for satisfying the $L$ condition is the following:
\[
\frac{(n-1)((n-1)+1)(2(n-1)+1)}{6} \cdot 3 = \frac{n(n-1)(2n-1)}{2}
\]
So we have a total of $\frac{n(n-1)(2n-1)}{2} + 4n^2$ CNF statements for an $n \times n$ grid.
 \end{proof}

\subsection{Solving Methods}
In this section, we detail and document the methods we have used to try and find a $21 \times 21$ grid.

\subsubsection{Brute Force}
This method briefly involves taking the entire CNF statement that corresponds with a $3$-colored $21 \times 21$ grid with no monochromatic $L$. This is the method detailed when finding the $20 \times 20$ grid. Code for generating the CNF statement as well as the code for parsing it into a $20 \times 20$ grid are listed in the appendix. The SAT solver used was CryptoMiniSAT \cite{soos}. After running for a month on ParKissat-RS on a 32GB memory virtual machine on my laptop, this procedure has not yielded satisfiability or unsatisfiability of the $21 \times 21$ grid \cite{ParKissat}.

\subsubsection{Triangulation}
This method involved generating the CNF code to check if there is a $21 \times 21$, $3$-colored right triangle with no monochromatic $L$. Then, we get our right triangle and create a CNF statement checking to see if there exists a $21 \times 21$ $3$-colored grid with the bottom half (main diagonal and all subdiagonals below) set to the values of the right triangle generated above. This method did generate satisfiable $21 \times 21$ right triangles; however, when setting these values in the bottom half of a $21 \times 21$ grid, the problem returned unsatisfiable. The code for generating the right triangle, parsing the CNF output for the satisfiable right triangle into CNF code, and generating the CNF code for the rest of the triangle are all in the appendix.

\subsubsection{The ``Right Column" Method}

If one takes a look at the $20 \times 20$ $3$-colored grid with no monochromatic $L$ on page $15$, they will notice that the rightmost column contains an uncanny number of ones. Our intuition (and a chi-squared test on the entries in this last column) tells us that one of the colors used will cluster on the rightmost column. Filling the right hand column with all ones for an $18 \times 18$ grid returned satisfiable; however, the SAT solver was slower to generate a satisfiable result than the normal CNF code. Running the $19 \times 19$ grid with the right column fixed returned satisfiable as well, but it took much longer to run than the normal $19 \times 19$ instance.

\subsubsection{Fixing Values}

For each satisfiable grid, there are $3! = 6$ congruent grids generated by simply swapping colors. If we fix the color of the entry in the first column and first row, then we still check each possible assortment of colors in a grid while cutting down on the computation time by $3$. Running the normal code for a $19 \times 19$ grid gave $308986$ conflicts before solving, while fixing the first value gave $115436$ conflicts before solving. From this result, we would think that fixing the first value cuts down on solving time by around a third. However, running the normal code for a $20 \times 20$ grid took 1072210 (approx. 7 mins) conflicts before solving on CryptoMiniSAT, and fixing the entry in the first row and column alone returned 7061550 conflicts in about 3 hours on CryptoMiniSAT \cite{soos}.

Fixing the values in the first row and column as well as the first row and second column to be different values would limit the total pool of grids to check for satisfiability but would also eliminate all congruent grids. However, like with the case above, fixing these values in this took much longer than 1072210 conflicts. Therefore, both of these methods of fixing values turn out to be slower for these larger grids.
\subsubsection{Reverse Diagonals}

Another pattern one can notice when looking at the $20 \times 20$ grid is that the colors cluster on the reverse diagonals (by this, we mean the diagonals in the grid that run from the bottom left corner to the top right). On the page below is a picture of the main reverse diagonal in a $5 \times 5$ grid for demonstration.

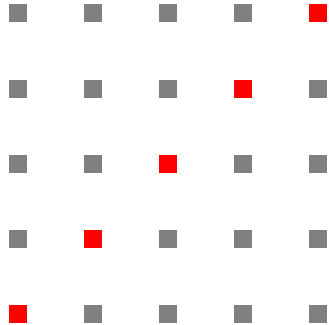
\begin{figure}[!h]
\centering
\begin{tikzpicture}
% a is the name of the node and A is the text inside the node/vertex  
    \node[fill=gray] (1) at (0,4) {};  
    \node[fill=gray] (2) at (1,4) {};  
    \node[fill=gray] (3) at (2, 4) {};  
    \node[fill=gray] (4) at (3, 4) {}; 
    \node[fill=red] (5) at (4, 4) {}; 
    \node[fill=gray] (6) at (0, 3) {}; 
    \node[fill=gray] (7) at (1, 3) {}; 
    \node[fill=gray] (8) at (2, 3) {}; 
    \node[fill=red] (9) at (3, 3) {}; 
    \node[fill=gray] (10) at (4, 3) {}; 
    \node[fill=gray] (11) at (0, 2) {}; 
    \node[fill=gray] (12) at (1, 2) {}; 
    \node[fill=red] (13) at (2, 2) {}; 
    \node[fill=gray] (14) at (3, 2) {}; 
    \node[fill=gray] (15) at (4, 2) {}; 
    \node[fill=gray] (16) at (0, 1) {}; 
    \node[fill=red] (17) at (1, 1) {}; 
    \node[fill=gray] (18) at (2, 1) {}; 
    \node[fill=gray] (19) at (3, 1) {}; 
    \node[fill=gray] (20) at (4, 1) {}; 
    \node[fill=red] (21) at (0, 0) {}; 
    \node[fill=gray] (22) at (1, 0) {}; 
    \node[fill=gray] (23) at (2, 0) {}; 
    \node[fill=gray] (24) at (3, 0) {}; 
    \node[fill=gray] (25) at (4, 0) {}; 
\end{tikzpicture}
\caption{A picture of the main reverse diagonal (colored red) on a $5 \times 5$ integer grid.}
\end{figure}
On the main reverse diagonal in our $20 \times 20$ grid, we have $12$ twos. Furthermore, each of the uppermost $4$ diagonals are comprised of a single value (the uppermost is only twos, the two below are only ones, and the one below those is only zeros). This gives us reason to think that one of our colors will cluster on the main diagonal of a satisfiable grid. However, fixing the values to $1$ in the second diagonal of a $19 \times 19$ grid returned a satisfiable result in $1015572$ conflicts, much slower than our original number of $308986$ conflicts without fixing anything. More importantly, fixing all the values in the 7th reverse diagonal of the $19 \times 19$ grid returned unsatisfiable in $6896$ conflicts. Thus, the approach of fixing reverse diagonals (at least in this manner) seems to not be very helpful.

Below is a lemma and theorem that may motivate the clustering of colors on reverse diagonals in satisfiable grids. We monochromatically color the reverse diagonals to essentially turn them into monochromatic arithmetic sequence and then use properties of van der Waerden numbers to establish bounds for $R_k(L)$.

\begin{lem}
    The number of diagonals running from bottom left to top right in an $n \times n$ grid is $2n-1$.
\end{lem}

\begin{proof}
    We first note that the points whose coordinates add up to the same value are in the same reverse diagonal. For example, $(1, 3), (2, 2), (3,1)$ would all be in the same diagonal (namely, the 3rd one from the top), and the middle diagonal consists of points whose coordinates add up to $n+1$. Since the sum of the coordinates of a point in an $n \times n$ grid can range from any value between $2$ (from point $(1,1)$) to $2n$ (to point $(n,n)$), there are $2n-1$ total possibilities for the sum of a point's coordinates, and thus $2n-1$ diagonals running from bottom left to top right.
\end{proof}
Now that we have covered our lemma, we can move on to our theorem. Note that the way we defined diagonals in the lemma (as groups of coordinates adding up to a certain number) is used in our proof.

For our next theorem, recall that $W(r, k)$ is the smallest value such that coloring the first $W(r, k)$ natural numbers with $r$ colors guarantees an arithmetic progression of length $k$.

\begin{theorem}
    $R_k(L) \geq \left\lfloor{\frac{W(k,3)}{2}}+1\right\rfloor$ 
\end{theorem}

\begin{proof}
    Consider a square grid with side length $\left\lfloor{\frac{W(k,3)}{2}}\right\rfloor$. From Lemma $11$ (number may change), an $n \times n$ grid will have $2n-1$ diagonals running from bottom left to top right. If $W(k, 3)$ is odd, then $\left\lfloor{\frac{W(k,3)}{2}}\right\rfloor = \frac{W(k,3)-1}{2}$. If $W(k, 3)$ is even, then $\left\lfloor{\frac{W(k,3)}{2}}\right\rfloor = \frac{W(k,3)}{2}$. So if $W(k, 3)$ is even, there are $2\left(\frac{W(k,3)}{2}\right)-1 = W(k,3)-1$ diagonals, and if $W(k, 3)$ is odd, there are $2\left(\frac{W(k,3)-1}{2}\right)-1 = W(k,3)-3$ diagonals. In each case, there are less than $W(k,3)$ diagonals. By definition, there exists the $3$-colored sequence $1, 2, \dots, W(k,3)-1$ that contains no monochromatic arithmetic progression of length $k$. Let $i$ in the sequence has some color $c_i$. Assign color $c_1$ to the first diagonal (that is, the diagonal whose points have coordinates that sum to $2$), $c_2$ to the second diagonal, and so on until we have covered every diagonal.

    Now, assume for purposes of contradiction that a monochromatic $L$ exists; that is, for some $(i, j)$, there exists $t \in \mathbb N$ such that $(i, j), (i+t, j), (i+t, j+t)$ is monochromatic. In this case, $(i, j)$ would have color $c_{i+j-1}$, $(i+t, j)$ would have color $c_{i+j+t-1}$, and $(i+t, j+t)$ would have color $c_{i+j+2t-1}$. However, if these colors are all the same, then $i+j-1, i+j+t-1, i+j+2t-1$ is an arithmetic progression of length $3$, which cannot happen given the way we colored the sequence $1, 2, \dots, W(k,3)-1$. Due to this contradiction, there exists a $k$ coloring of a $\left\lfloor{\frac{W(k,3)}{2}}\right\rfloor \times \left\lfloor{\frac{W(k,3)}{2}}\right\rfloor$ grid, and so $R_k(L) \geq \left\lfloor{\frac{W(k,3)}{2}}+1\right\rfloor$ as desired.
\end{proof}
While this proof does not give us any new bounds (it gives that $R_3(L) \geq \left\lfloor{\frac{W(3,3)}{2}}+1\right\rfloor = \left\lfloor{\frac{27}{2}}+1\right\rfloor = 14)$, it does give us insight as to how grids with high numbers of $L$s may be structured.

\subsection{Table of Solving Methods}
\begin{figure}[!h]
\begin{tabular}{c|c|c|c|c|c}
   &  Brute Force & Triangulation & \makecell{Fixing First \\ Value} & \makecell{Fixing First \\ Two Values} & \makecell{Fixing \\ Diagonals} \\ \hline
   $18 \times 18$ & Satisfiable & N/A & N/A & N/A & N/A\\ \hline
   $19 \times 19$ & \makecell{Satisfiable \\ (308986 \\ conflicts)}& N/A & \makecell{Satisfiable \\ (115436 \\ conflicts)}& \makecell{Satisfiable \\ (44246 \\ conflicts)}& Unsatisfiable \\ \hline
   $20 \times 20$ & \makecell{Satisfiable \\ (1072210 \\ conflicts)}& Unsatisfiable & Ran too long & Ran too long & Unsatisfiable\\
\end{tabular}
\caption{A table of the solving methods mentioned in $5.3.1$ and their outcomes on $18 \times 18$, $19 \times 19$, and $20 \times 20$ grids. N/A implies that a specific method on a specific grid was not run, usually because running it on a $20 \times 20$ grid prove to be inefficient or took longer than the standard ``brute force" method.}
\end{figure}
\section{Open Problems}

We close this paper with some open problems regarding the ``L" problem.
\begin{itemize}
    \item Can interval sequencing (as detailed in Section 2.2 before Theorem 4) be used to further improve the upper bound?
    \item Can properties of diagonals below the main diagonal and subdiagonal of length $n-1$ be used to improve the upper bound?
    \item What are upper and lower bounds for $R_4(L)$? $R_k(L)$?
    \item Though not found, we speculate that a $3$-coloring of $[22] \times [22]$ with no monochromatic $L$ exists. Try to find one, perhaps by using SAT solvers or AI/ML techniques.
\end{itemize}

\paragraph{Acknowledgements.} We would like to acknowledge and thank the referees of the \textit{Integers} journal for their proofs for the upper bounds of $772$ and $493$ as well as their consistent, helpful feedback. We would also like to acknowledge and thank Michael Blackmon and Jonad Pulaj at Davidson College for computing assistance.

%Other documents that I have found: \\
%Papers on Integer Grid:
%\href{https://www.cs.umd.edu/~gasarch/reupapers/ramseyandsat.pdf}{https://www.cs.umd.edu/~gasarch/reupapers/ramseyandsat.pdf} \\
%
%\href{https://personal.lse.ac.uk/corsten/gridramsey.pdf}{https://personal.lse.ac.uk/corsten/gridramsey.pdf} \\
%\href{http://math.colgate.edu/~integers/w62/w62.pdf}{http://math.colgate.edu/~integers/w62/w62.pdf} \\
%\href{http://math.colgate.edu/~integers/graham20/graham20.pdf}{http://math.colgate.edu/~integers/graham20/graham20.pdf} \\
%\href{https://arxiv.org/pdf/1005.3750.pdf}{https://arxiv.org/pdf/1005.3750.pdf} \\
%\href{https://link.springer.com/content/pdf/10.1007/3-540-33700-8.pdf}{https://link.springer.com/content/pdf/10.1007/3-540-33700-8.pdf} pages $129-132$\\
%
%Interesting Articles from Integers: \\
%\href{http://math.colgate.edu/~integers/w62/w62.pdf}{http://math.colgate.edu/~integers/w62/w62.pdf} \\
%\href{http://math.colgate.edu/~integers/graham20/graham20.pdf}{http://math.colgate.edu/~integers/graham20/graham20.pdf} \\
%\href{http://math.colgate.edu/~integers/graham11/graham11.pdf}{http://math.colgate.edu/~integers/graham11/graham11.pdf} \\
%
%Dr. Yerger’s Paper: \\
%\href{https://arxiv.org/pdf/1407.5122.pdf}{https://arxiv.org/pdf/1407.5122.pdf}\\
%
%Other Papers: \\
%\href{https://people.math.ethz.ch/~sudakovb/grid-ramsey-talk.pdf}{https://people.math.ethz.ch/~sudakovb/grid-ramsey-talk.pdf} \\
%\href{https://arxiv.org/pdf/1711.08076.pdf}{https://arxiv.org/pdf/1711.08076.pdf} \\
%\href{https://users.renyi.hu/~p\_Erdös/1970-03.pdf}{https://users.renyi.hu/~p\_Erdös/1970-03.pdf} \\
%\href{https://math.mit.edu/~fox/paper-gridramsey.pdf}{https://math.mit.edu/~fox/paper-gridramsey.pdf} \\


\begin{thebibliography}{1}\footnotesize

\bibitem{balaji2021schur} V. Balaji, A. Lott, and A. Rice, Schur's theorem in integer lattices, {\it Integers} {\bf 22} (2022), \#A62.

\bibitem{canackisat} B. Canacki, H. Christenson, R. Fleischman, N. McNabb, and D. Smolyak, On sat solvers and {R}amsey-type numbers, preprint,  {\tt arXiv: 2312.01159}.

\bibitem{ParKissat} S. Cai, Z. Chen, and X. Zhang, ParKissat, accessible at \url{https://github.com/shaowei-cai-group/ParKissat-RS}.

\bibitem{dumitrescu2004coloring} A. Dumitrescu and R. Radoicic, On a coloring problem for the integer grid, {\it Contemp. Math.} {\bf 342} (2004), 67–74.

\bibitem{graham1991ramsey} R. L. Graham, B. L. Rothschild, and J. H. Spencer, {\it Ramsey Theory}, John Wiley \& Sons, New York, 1991.

\bibitem{graham2006monochromatic} R. L. Graham and J. Solymosi, Monochromatic equilateral right triangles on the integer grid, in {\it Topics in Discrete Mathematics: Dedicated to Jarik Ne{\v{s}}et{\v{r}}il on the Occasion of his 60th Birthday}, Springer Berlin Heidelberg, 2006.

\bibitem{manske} J. Manske, \textit{Coloring in Extremal Problems in Combinatorics}, Iowa State University, Ames, 2010.

\bibitem{shkredov2022tilted} I. D. Shkredov and J. Solymosi, Titled corners in integer grids, in {\it Number Theory and Combinatorics: A Collection in Honor of the Mathematics of Ronald Graham}, de Gruyter, 2022.

\bibitem{soos} M. Soos, Cryptominisat, accessible at \url{https://msoos.github.io/cryptominisat_web/}.

\bibitem{witt} E. Witt, Ein kombinatorischer Satz der Elementargeometrie, {\it Mathematische Nachrichten} {\bf 6} (1952), 261-262.

\end{thebibliography}
\end{document}